\documentclass[a4paper,11pt]{article} 
\makeatletter

\@addtoreset{equation}{section}
\makeatother
\usepackage[dvipdfmx]{graphicx}
\usepackage[dvipdfm]{hyperref}
\usepackage{amsmath}
\usepackage{amssymb}
\usepackage{latexsym}
\usepackage{amsthm}
\newtheorem{Thm}{Theorem}[section]

\newtheorem{Lem}[Thm]{Lemma}
\newtheorem{Prop}[Thm]{Proposition}

\newtheorem{Cor}[Thm]{Corollary}

\theoremstyle{definition}

\newtheorem{Rem}[Thm]{Remark}
\newtheorem{Exa}[Thm]{Example}

\begin{document}

\title{On regularity for de Rham's functional equations}
\author{Kazuki Okamura\footnote{{\sc Research Institute for Mathematical Sciences, Kyoto University, Kyoto 606-8502, JAPAN.} e-mail : \texttt{kazukio@kurims.kyoto-u.ac.jp} }}
\maketitle

\begin{abstract}
We consider regularity for solutions of a class of de Rham's functional equations. 
Under some smoothness conditions of functions consisting the equation,  
we improve some results in Hata (Japan J. Appl. Math. 1985).    
Our results are applicable to some cases that the functions consisting the equation are non-linear functions on an interval, 
specifically, polynomials and  linear fractional transformations.  
Our results imply singularity of some well-known singular functions, in particular, Minkowski's question-mark function, and, some small perturbed functions of the singular functions.  
\end{abstract}

\section{Introduction and Main results}  

De Rham \cite{dR}\footnote{An English translation of \cite{dR} is included in Edger \cite{E}.} considered a certain class of functional equations. 
Solutions of de Rham's functional equations give parameterizations of some self-similar sets such as the Koch curve and the P\'olya curve, etc.   
Some singular functions\footnote{A continuous increasing function on $[0,1]$ whose derivative is zero Lebesgue-a.s.} 
such as the Cantor, Lebesgue, etc. functions are solutions of such functional equations. 
The survey by Kairies \cite{K} and the monograph by Kannappan \cite[Chapter 14.4]{Kan} 
study functional equations including de Rham's ones.

Let $X$ be a metric space and consider the following functional equation for $G : [0,1] \to X$ : 
\begin{equation}\label{Def}
G(t) = f_i (G(mt - i)),  \  \frac{i}{m} \le t \le \frac{i+1}{m}, \, 0 \le i \le m-1.  
\end{equation}
Here $f_i : X \to X$, $0 \le i \le m-1$, are weak contractions such that $f_{i-1}(\text{Fix}(f_{m-1})) = f_i(\text{Fix}(f_0))$ for any $i$, and, $\text{Fix}(f_i)$ denotes a fixed point of $f_i$.  
We mainly follow Hata \cite[Sections 6 and 7]{Ha} for a framework of de Rham's functional equations\footnote{\cite{Ha} considers a more general setting than ours.}.   
\cite{dR} considers the case that $X = \mathbb{R}^2$ and $m = 2$ mainly.   
\cite[Sections 6 and 7]{Ha} shows that a unique continuous solution $G$ of (\ref{Def}) exists, and,  
under some conditions for $X$ and $f_i$ s, 
some regularity results such as the H\"older continuity, variation, and differentiability of the solution $G$ are obtained.         

In this paper, we improve some regularity results of $G$ obtained in \cite[Section 7]{Ha},  
if $f_i$ s are Fr\'echet differentiable and 
the Fr\'echet derivative $Df_i(x)$ of $f_i$ at $x$ is uniformly continuous with respect to $x$.   
See (A-i) to (A-iv) in the following subsection for precise assumptions. 
It seems that these assumptions are natural and not restrictive.

Our main results are applicable to some cases to which 
the results in \cite[Section 7]{Ha} are not applicable.     
\cite[(7.3) in Theorem 7.3]{Ha} corresponds to $\prod_{i=0}^{m-1} \max_{x \in X} \|Df_i (x)\| < m^{-m}$, and, 
\cite[(7.6) in Theorem 7.5]{Ha} corresponds to $\prod_{i=0}^{m-1} \min_{x \in X} \|Df_i (x)\| > m^{-m}$ 
If $f_i$ s are not linear, 
then, it can occur that $\prod_{i=0}^{m-1} \max \|Df_i\| \ge m^{-m} \ge \prod_{i=0}^{m-1} \min \|Df_i\|$.    

We state some examples of the cases that $f_i$ s are not linear and our main results are applicable.     
Previously, in \cite{O1} and \cite{O2}, the author considered some regularities of $G$ 
if $X = [0,1]$, $m=2$,  and, $f_i$ s are certain linear fractional transformations on $X$. 
$\max \|Df_0\| \max \|Df_1\| \ge 1/4 \ge \min \|Df_0\| \min \|Df_1\|$ can occur. 
By \cite{dR}, 
the inverse function of Minkowski's question-mark function is the solution of (\ref{Def}) for the case that $X = [0,1]$, $m=2$, $f_0(x) := x/(x+1)$ and $f_1(x) := 1/(2-x)$.          
Then, $\max \|Df_0\| \max \|Df_1\|  = 1$ and $\min \|Df_0\| \min \|Df_1\| = 1/4$.  
This case is in the framework of \cite{O2}.  
See Example \ref{Exa-2} (iii) for details.  
Our main results are also applicable to some examples which are even {\it outside} the framework of \cite{O1} and \cite{O2}.   
In Example \ref{Exa-2} (i) (resp. (ii)), 
we consider the case that $X = [0,1]$, $m=2$, and, $f_0$ is a polynomial with degrees $2$ (resp. $3$).     
We also consider an example such that $X = \mathbb{R}^2$ in Example \ref{exa-nonlinear}.    

If $X = [0,1]$, then, 
we can show singularity for some well-known singular functions such as the Cantor, Lebesgue and Minkowski functions etc.,  
by regarding them as solutions of a certain class of de Rham's functional equations and considering regularity of the solutions.        
Thanks to the approach, 
we can also show singularity of some slightly {\it perturbed} functions of the singular functions.

\subsection{Framework and main results}
 
We mainly follow \cite[Sections 6 and 7]{Ha} for notation.  
Let $X$ be a closed subset of a separable\footnote{Separability of $E$ is not assumed in \cite{Ha}, but here we assume it.} Banach space $E$ such that the interior of $X$ is non-empty.  
Let $B(E, E)$ be the set of linear bounded transformations on $E$.  
Let $| \cdot |$ be the norm of $E$.   
Let $m \ge 2$.  
Let $f_i, i = 0,1, \dots, m-1$, be functions from $X$ to $X$ 
such that \\
(A-i) $f_i$ is a weak contraction on $X$ for each $i$. \\
(A-ii) The $D$-conditions in \cite{Ha} hold. 
Specifically, 
$f_i(\text{Fix}(f_{m-1})) = f_{i+1}(\text{Fix}(f_0))$, $0 \le i \le m-2$. 

In this paper the following conditions are assumed.  \\
(A-iii) (Differentiability) The Fr\'echet derivative of $f_i$ at $x \in X$, which is denoted by $D f_i (x) \in B(E, E)$, exists for any $x \in X$. \\
(A-iv) (Continuity of the derivative) For each $i$, $D f_i(x)$ is uniformly continuous on $X$ with respect to the operator norm $\| \cdot \|$ of $B(E, E)$.  

In this paper, a map $f : X \to X$ is called {\it linear}  
if there is a $D\in B(E, E)$ such that $Df(x) = D$ holds for any $x \in X$; otherwise it is called {\it non-linear}.      
Since $f_i$ s are weak contractions, 
we have that $\| D f_i (x) \| \le 1$ for any $x$ and $i$.    
Thanks to \cite[Theorem 6.5]{Ha}\footnote{This holds if $X$ is a complete separable metric space as in \cite[Section 6]{Ha}. 
However, since we consider regularity of $G$, as in \cite[Section 7]{Ha}, 
we restrict $X$ to a closed subset of a (separable) Banach space.}, 
there exists a unique continuous solution $G$ of (\ref{Def}) such that $G(0) = \text{Fix}(f_0)$ and $G(1) = \text{Fix}(f_{m-1})$.      
Since we assume (A-iii) and (A-iv), 
our framework is less general than the framework of \cite{Ha}.
However, our framework contains the cases that $f_i$ s are linear and the framework of \cite{O1} and some parts of \cite{O2}.

Let $m \ge 2$.  
$\log_m$ denotes the logarithm with base $m$ and let $\log_m (0) := -\infty$ and $\log_m (+\infty) = +\infty$.     
Let $D_m := \cup_{n \ge 1} \{i/m^n : 0 \le i \le m^n-1\}$.    
Let $\ell$ be the Lebesgue measure on $[0,1)$.    
Let $\text{Lip}(g)$ be the Lipschitz constant of a function $g : X \to X$.

Let $t = \sum_{n \ge 1} A_n(t)/m^n$
be the $m$-adic expansion of $t \in [0,1)$.  
We always assume that the number of $n$ such that $A_n(t) \ne m-1$ is infinite. 
Let $t_n := \sum_{k = 1}^{n} A_k(t)/m^k$.  
Let\footnote{$A_1$ corresponds to $A$ in \cite[Section 6]{Ha}, but there is a slight difference. $H$ corresponds to $H$ in \cite[Section 6]{Ha}.} $Ht := \sum_{n \ge 1} A_{n+1}(t)/m^n$.  
Let $M_{n}(t) := |G(t_n + m^{-n}) - G(t_n)|$.   
Let 
\[ \alpha := \int_{0}^{1} -\log_m \left\| D f_{A_1(t)} (G(Ht)) \right\| \ell (dt), \text{ and, } \]
\[\beta := \int_{0}^{1} \log_m \left\| (D f_{A_1(t)} (G(Ht)))^{-1} \right\| \ell (dt). \]
(If $Df_{i}(x)$ is not invertible, then, we let $\| (D f_{i} (x))^{-1} \|$ be $+\infty$.)
Since $E$ is separable and $\left|Df_i(x)(v)\right|$ is continuous with respect to $x$ for any $v \in E$, we have that 
$\|(Df_{i}(x))^{-1}\|^{-1}$ is a continuous function with respect to $x$.
Hence, $\| (D f_{A_1(t)} (G(Ht)))^{-1} \|$ is measurable as a function of $t$.     

We have that $0 \le \alpha \le \beta$.  
$\alpha < \beta$ can occur (See Remark \ref{ineq} for details.).  
We remark that if $\text{Fix}(f_0) = \text{Fix}(f_{m-1})$, then, $G$ is constant, and hence, $\alpha = \beta = +\infty$.

\begin{Thm}\label{Basic}
We have that 
\begin{equation}\label{lower}
\liminf_{n \to \infty} \frac{-\log_m |M_{n}(t)| }{n} \ge \alpha,  \, \text{ $\ell$-a.s. $t$, and, }  
\end{equation}
\begin{equation}\label{upper}  
\limsup_{n \to \infty} \frac{-\log_m |M_{n}(t)| }{n} \le \beta,  \, \text{ $\ell$-a.s. $t$. }  
\end{equation}  
\end{Thm}

If $\alpha > 1$ or $\beta < 1$, then, singularity for the solution $G$ occurs. 
Precisely,  
\begin{Thm}[Differentiability]\label{Dif}  
We have that\\
(i) If $\alpha > 1$, the Fr\'echet derivative of the solution $G$ is zero, $\ell$-a.s.\\  
(ii) If $\beta < 1$, then, the Fr\'echet derivative of $G$ does not exist, $\ell$-a.s.  
\end{Thm}

Under the assumptions (A-iii) and (A-iv),     
we can weaken the assumptions of \cite[Theorems 7.3 and 7.5]{Ha}. 
Specifically, if $\prod_{i=0}^{m-1} \max_{x} \|D f_i(x)\| < m^{-m}$, that is, \cite[(7.3)]{Ha} holds, then, $\alpha > 1$. 
If $\prod_{i=0}^{m-1} \min_{x} \|D f_i(x)\| > m^{-m}$, that is, \cite[(7.6)]{Ha} holds,  then, $\beta < 1$.  

If $X = [0,1]$ and the solution $G$ is an absolutely continuous function, then, $\alpha$ must be equal to $1$. 
However, there is an example such that $\alpha = 1$ but the solution $G$ is not differentiable almost everywhere.  
See Example \ref{exa-1} (iii) for details.   

Calculating $\alpha$ and $\beta$ is not easy, 
because the integrals in the definitions of $\alpha$ and $\beta$ contain the solution $G$, which can be a fractal function.  
However, we can give satisfiable estimates for the integrals for some cases including the cases that $f_i$ s are linear.  
See Section 2 for such estimates.    

Theorem \ref{Basic} implies the following : 
\begin{Cor}[Variation]\label{Var}
Assume that $\beta < +\infty$.  
If $p < 1/\beta$, then, 
the solution $G$ is not of bounded $p$-variation. 
\end{Cor}

This corollary corresponds to \cite[Theorem 7.2]{Ha}. 
However, there is an example such that \cite[(7.2) in Theorem 7.2]{Ha} fails but $p < 1/\beta$ holds. 
See Example \ref{Exa-2} (iii) for details.

If the solution $G$ is of bounded ($1$-)variation, then, $\beta \ge 1$.   
It depends on settings whether $G$ is of bounded $1/\beta$-variation. 
There is an example such that $G$ is of bounded $1/\beta$-variation, and, on the other hand, 
there is also an example such that $G$ is not of bounded $1/\beta$-variation. 
See Example \ref{exa-1} (iii) for details.

This paper is organized as follows. 
In Section 2.1 (resp. 2.2), we consider examples for one-(resp. two-)dimensional cases. 
In Section 2.3, we discuss perturbations of solutions. 
In Section 3, we give proofs of Theorems \ref{Basic} and \ref{Dif}.  
In Section 4, we state some remarks.


\section{Examples}  

\subsection{One-dimensional cases}  

In this subsection we let $X := [0,1]$ and $E := \mathbb{R}$. 
If each $f_i$ is an increasing function, 
then, by noting (\ref{Def}), the solution $G$ is increasing, and,   
$G$ is the distribution function of a probability measure $\mu$ on $[0,1]$. 

The solution $G$ is a singular continuous function 
if and only if the measure $\mu_G$ whose distribution function is $G$ is singular with respect to the Lebesgue measure. 
(See Riesz and Sz.-Nagy \cite[Section 25]{RSz} for details.)   

If $X \subset \mathbb{R}$, then since $\|Df_{i}(x)\| = \|Df_{i}(x)^{-1}\|^{-1}$,  we have $\alpha = \beta$.  

\begin{Exa}[linear cases]\label{exa-1}  
We have \\
(i) (The Cantor functions)  
Let $m = 3$ and $f_0(x) = x/2$ and $f_1(x) = 1/2$ and $f_2(x) = (x+1)/2$.   
Then, the solution $G$ is the Cantor function and $\alpha = +\infty$. \\
(ii) (The Bernoulli measures; the Lebesgue singular functions)
If $f_{i}(x) = a_i x$ for some $a_{0}, \dots, a_{m-1} \in (0,1)$ satisfying $\sum_{i=0}^{m-1} a_i = 1$.   
Then, $\alpha = \beta = -\sum_{i=0}^{m-1} (\log_m a_i) / m$.   
$\alpha \ge 1$, and,  
$\alpha = 1$ if and only if $a_{i} = 1/m$ for each $i$.  
Let $\mu_G$ be the probability measure on $[0,1]$ whose distribution function is the solution $G$. 
Then, $\mu_G$ is absolutely continuous if $a_{i} = 1/m$ for each $i$, and, singular otherwise. 
If $m = 2$ and $(a_1, a_2) \ne (1/2, 1/2)$, then, $G$ is the Lebesgue singular function. \\  
(iii) (Functions in Okamoto \cite{Okamoto-1} and Okamoto and Wunsch \cite{OW})    
Let $0 < a < b < 1$.    
If $m = 3$, $f_0(x) = ax$, $f_1(x) = a + (b-a)x$, and, $f_2(x) = (1-b)x + b$, then, 
the solution $G$  is identical with  $f_{a, b}$ in \cite{OW}.     
If $a = b = 1/2$, $f_{a, b}$ is the Cantor function. 
\cite[Theorem 5]{OW} states that $f_{a,b}$ is continuous, strictly increasing and singular if $(a, b) \ne (1/3, 2/3)$.  

We have $\alpha = -\log_3 (a (b-a) (1-b)) / 3$ and hence $\alpha > 1$ if and only if $(a, b) \ne (1/3, 2/3)$.       
Therefore, Theorems \ref{Basic} and \ref{Dif} immediately imply that $f_{a,b}$ is singular if $(a, b) \ne (1/3, 2/3)$.      
(It is easy to see that $f_{a,b}$ is continuous and strictly increasing.)   

We remark that even if $\alpha = 1$, $G$ can be a non-smooth curve.   
Let $a_0 \in (1/2, 1)$ be a unique parameter $a$ such that $\alpha = 1$ for $f_{a_0, 1-a_0}$.   
Then, by using Kobayashi \cite{Ko}, $f_{a_0, 1 - a_0}$ is non-differentiable a.e. and hence $f_{a_0, 1 - a_0}$ is not absolutely continuous. 
It is easy to see that $G$ is not of bounded $1$-variation. 
\end{Exa}

The following examples are outside  the framework of \cite{O1}.  

\begin{Exa}[non-linear cases]\label{Exa-2}
We have \\
(i) 
Let $f_0(x) := x^{2}/2$ and $f_1(x) := (x + 1)/2$.     
We have that $\text{Lip}(f_0)\text{Lip}(f_1) = 1/2 > 1/4$ and $\min_{x}|Df_{0}(x)| \min_{x}|Df_{1}(x)|  = 0 < 1/4$.  
This does not satisfy the assumption of \cite[Theorem 7.3]{Ha}. 
We can show that $\alpha = +\infty$. 
See Figure 1 below for the graph of $G$. \\    
(ii)   
Let $f_0(x) := x^3/4 + x/12$ and $f_1(x) := (2x + 1)/3$.     
We have that $\text{Lip}(f_0)\text{Lip}(f_1) = 5/9 > 1/4$ and $\min_{x}|Df_{0}(x)| \min_{x}|Df_{1}(x)|  = 1/18 < 1/4$.  
This does not satisfy the assumption of \cite[Theorem 7.3]{Ha}. 
\[ \alpha = \beta = \frac{1}{2} \int_{0}^{1} - \log_2 \left(\frac{G(t)^2}{2} + \frac{1}{18}\right) \ell (dt) \ge \frac{1}{4} \log_2 \frac{81}{5} >  1. \] 
By using Theorem \ref{Dif}, $DG(t) = 0$, $\ell$-a.s. $t$.  
See Figure 2 below for the graph of $G$.\\
(iii) (The inverse function of Minkowski's question-mark function)     
Let $f_0(x) := x/(x+1)$ and $f_1(x) := 1/(2-x)$.  
They are weak contractions, and, $\text{Lip}(f_0) = \text{Lip}(f_1) = 1$.  
Let $G$ be the solution of (\ref{Def}). 
Since $G$ is continuous and strictly increasing, $G(0) = 0$ and $G(1) = 1$, 
we see that $\ell(0 < G < 1) > 0$. 
Hence, 
\[ \alpha = \beta = \int_{0}^{1} \log_2 (2 + G(t) - G(t)^2) \ell (dt) > 1. \]  
Theorem \ref{Dif} implies that $\mu_G$ is singular. 

\[ \text{Lip}(f_0^{-1})^{-p} + \text{Lip}(f_1^{-1})^{-p} = \min_{x \in [0,1]} |Df_0(x)|^p + \min_{x \in [0,1]} |Df_1(x)|^p > 1. \]
holds, that is, the assumption of \cite[(7.2) in Theorem 7.2]{Ha} holds, if and only if $p < 1/2$. 
\cite[Theorem 7.2]{Ha} implies that $G$ is not of bounded $p$-variation if $p < 1/2$.     
Since $\alpha < 2$, $1/2 \le p < 1/\alpha < 1$ can occur.  
By using Corollary \ref{Var}, $G$ is not of bounded $p$-variation if $p < 1/\alpha$.

By using this and \cite[Lemma 4.2]{O1},
the inverse function of $G$, which is identical with Minkowski's question-mark function,  
is also a singular function.  
We can also show the some small perturbed functions of Minkowski's function are also singular. 
See Example \ref{perturb} (ii) below. 
\end{Exa}

We give some comments about Example \ref{Exa-2} (iii).  

\begin{Rem}
(i) Denjoy \cite{De} and Salem \cite{Sa} considered singularity of Minkowski's function. 
In order to show singularity, 
\cite{De} uses an expression of the function by continued fractions, and, \cite{Sa} uses a geometric construction of the function. 
See Parad\'is, Viader and Bibiloni \cite[Section 1]{PVB} for details. 
This function is often defined by using continued fractions. 
\cite{dR} states this by using functional equations and our approach is investigating regularity of the solution of the functional equations in \cite{dR}.\\
(ii) Since $\text{Lip}(f_0) = \text{Lip}(f_1) = 1$,  
this case does not satisfy \cite[(A3) in Section 1]{O1}.
However, we can apply the technique in the proof of the author \cite[Lemma 3.5]{O2},  
because $0$ is not a fixed point of $x \mapsto x+1$ or $x \mapsto -1/(2+x)$.
The approach in \cite[Lemma 3.5]{O2} is different from the one in this paper. 
\cite[(A3) in Section 1]{O1} assures a stronger result than singularity, specifically, the Hausdorff dimension of $\mu_{G}$, $\dim_H(\mu_G)$ is strictly smaller than $1$.    
\end{Rem}

\begin{figure}[htbp]
\begin{minipage}{0.5\hsize}
\includegraphics[width = 6cm, height = 6cm]{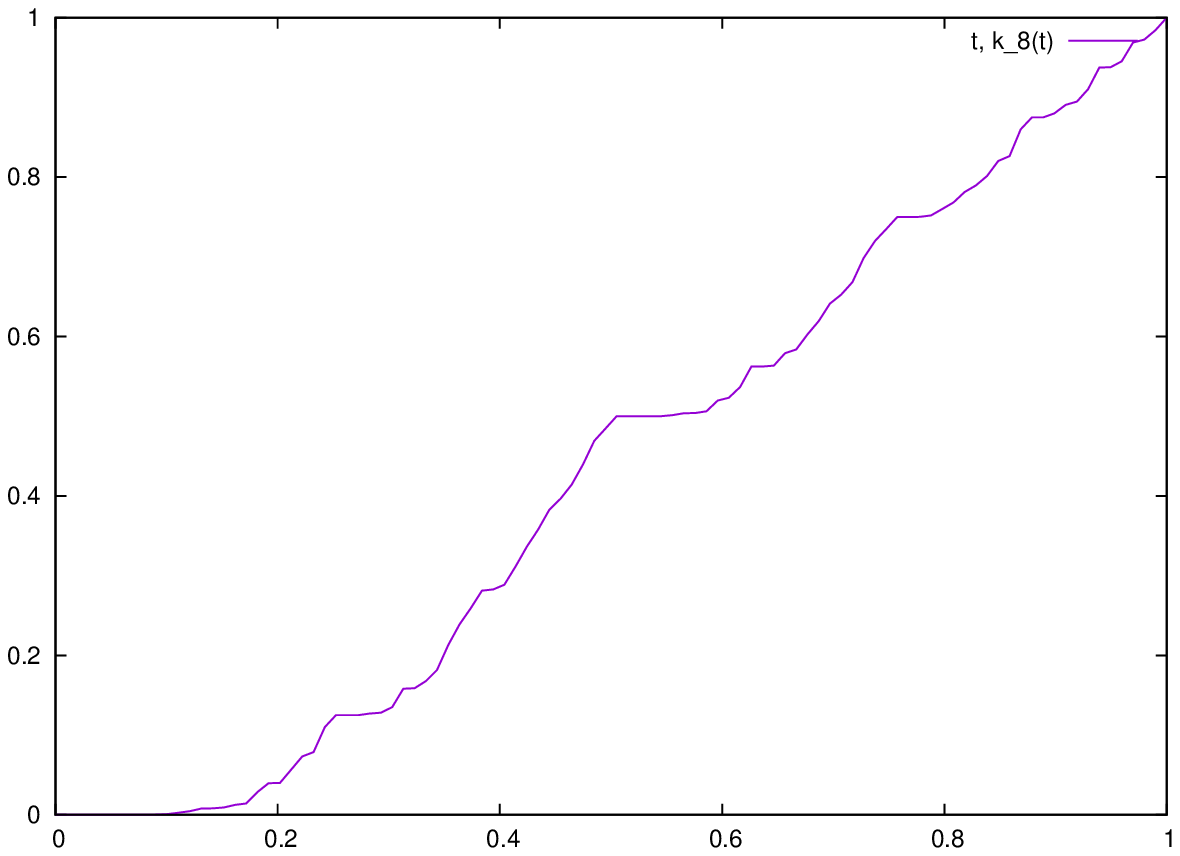}
\caption{\small $f_0(x) = x^{2}/2$ and \, \, \, \, \, \, \, \, $f_1(x) = (x + 1)/2$. }   
\end{minipage} 
\begin{minipage}{0.5\hsize}
\includegraphics[width = 6cm, height = 6cm]{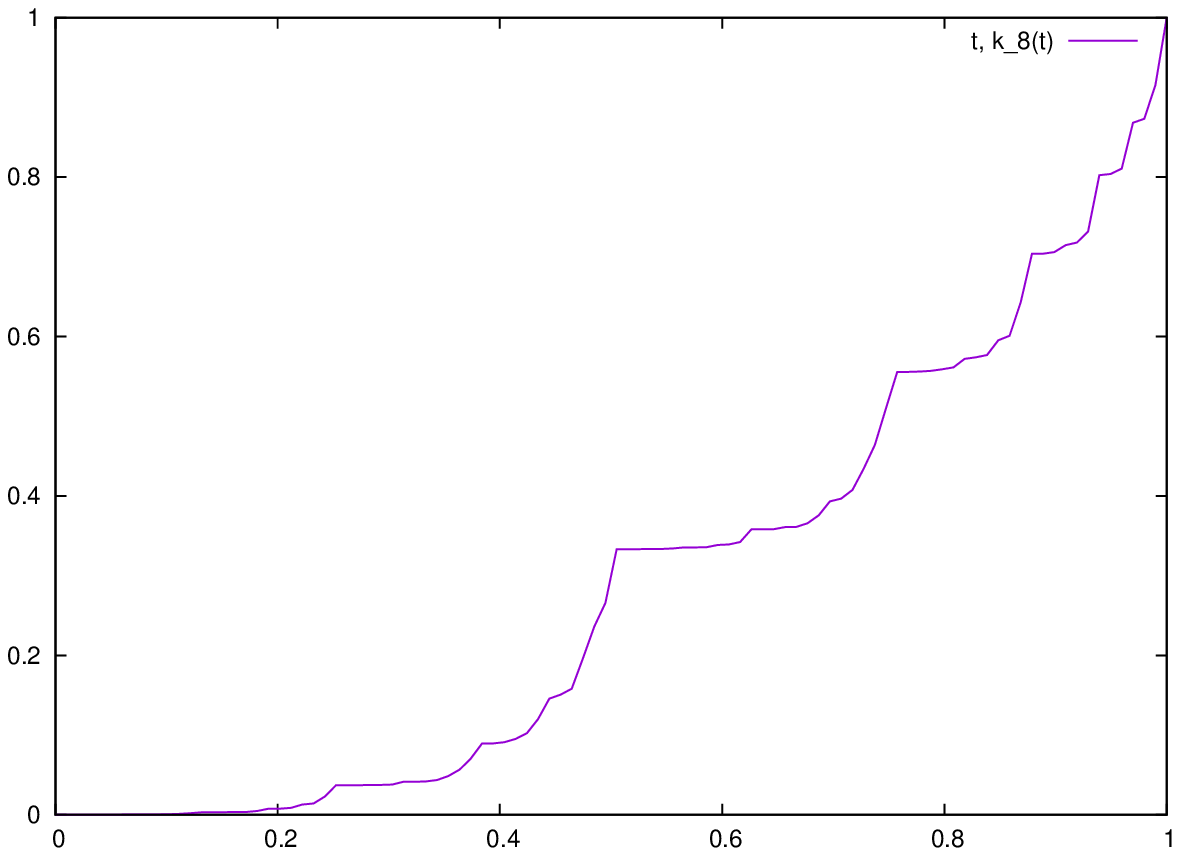}
\caption{\small $f_0(x) = x^3/4 + x/12$ and $f_1(x) = (2x + 1)/3$.  }   
\end{minipage}
\end{figure}

\subsection{Two-dimensional cases}  

In this subsection, we let $X = \mathbb{R}^2$, which is here identified with $\mathbb{C}$.
$i$ denotes the imaginary unit. $\overline{z}$ denotes the complex conjugate of $z \in \mathbb{C}$.                
We remark that if a function $f : \mathbb{C} \to \mathbb{C}$ is complex differentiable at $z$, then, $|f^{\prime}(z)| = \|Df(z)\| = \|(Df(z))^{-1}\|^{-1}$.    

First, we focus the case that both $f_0$ and $f_1$ are linear maps.   

\begin{Exa}[\text{\cite{dR}, \cite[Example in Section 6]{Ha}}]\label{exa-linear}
Let $\eta \in \mathbb{C}$ such that $|\eta| \le 1$ and $|1-\eta| \le 1$.  
Let $f_0(z) := (1-\eta)\overline{z}$ and $f_1(z) := (1- \eta)\overline{z} + \eta$.   
Then, $\alpha = \beta = -\log_2 |\eta(1-\eta)| / 2$.    
Theorem \ref{Dif}  implies that if $|\eta(1-\eta)| > 1/4$, then, the solution $G$ is not Fr\'echet differentiable, $\ell$-a.s. and, 
if $|\eta(1-\eta)| < 1/4$, then, the Fr\'echet derivative $DG$ is zero, $\ell$-a.s.  
\end{Exa}

We consider a case that neither $f_0$ nor $f_1$ is a linear map. 

\begin{Exa}\label{exa-nonlinear} 
Let $f_0(x) := z^2 (z + i) / (4(1+i))$ and $f_1(z) := (3z+1)/4$.  
Then, the complex derivatives of them are given by 
\[ f_0^{\prime}(z) = \frac{3z^2+2iz}{4(1+i)}, \text{ and, } f_1^{\prime}(z) = \frac{3}{4}, \, \, z \in \mathbb{C}. \]  
Therefore, 
\[ \text{Lip}(f_0) < 1,  \,  \text{Lip}(f_1) < 1,  \text{ and, } \text{Lip}(f_0)\text{Lip}(f_1) > 1/4,  \, \text{ on } \{z \in \mathbb{C} : |z| \le 1\}. \]

Since $|f_i| \le 1$ on $\{z : |z| \le 1\}$,     
the solution $G$ satisfies that $|G(t)| \le 1$ for any $t \in [0,1]$.    
By using (\ref{Def}), we have that 
$|G(t)| \le \sqrt{2}/4$ if $t \le 1/2$, and, 
$|G(t)| \le 2^{-5}$ if $t \le 1/4$. 
By using them,    
\begin{align*} \alpha &=  \frac{1}{2} \int_{0}^{1} -\log_2 \frac{9|G(t)| (3|G(t)| + 2)}{16 \sqrt{2}} \ell (dt) \\  
&\ge \frac{9}{4} - \frac{\log_2 15}{2} - \frac{1}{2}\int_0^1 \log_2 |G(t)| \ell (dt) > 1.  
\end{align*}

\begin{figure}[htbp]
\centering
\includegraphics[width = 12cm, height = 4cm]{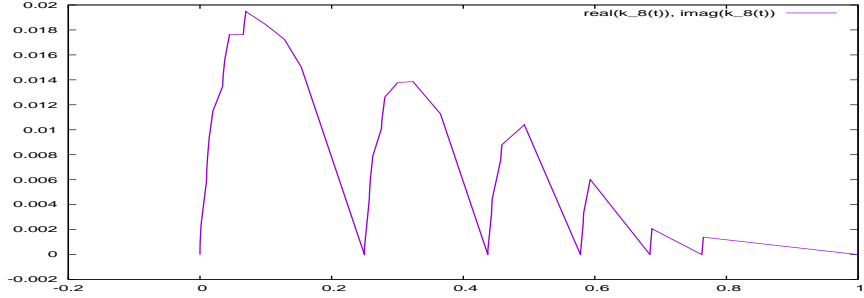}
\caption{Rough image of $G$ for $f_0(z) = z^2 (z + i) / (4(1+i))$ and $f_1(z) = (3z+1)/4$.  }
\end{figure}
\end{Exa}

\begin{Rem}[$\alpha < \beta$ can occur]\label{ineq}
Let 
$f_{0}(x+ iy) = x/2 + iy/3$ and $f_1(x+iy) = (x+1)/2$.     
Then, for any $z \in \mathbb{C}$, 
$\|Df_0(z)\| = 1/2$, $\|(Df_0(z))^{-1}\|^{-1} = 3$, $\|Df_1(z)\| = 1/2$, and, $\|(Df_1(z))^{-1}\|^{-1} = 0$.  
(We remark that none of $f_i$ are complex differentiable.)    
Hence, $\alpha = 1$ and $\beta = +\infty$.   
\end{Rem}

\subsection{Continuity of solutions} 

In this subsection, we consider perturbations of $f_i$ s.    
Fix $X$ and $m$.  
Denote $G, \alpha, \beta$ by $G_n, \alpha_n, \beta_n$, if $f_i = f_{n, i}$ for each $i$.    

\begin{Prop}[Continuity of solutions]\label{conti}   
We have that \\
(i) If $f_{n, i} \to f_{i}$, $n \to \infty$,  uniformly on $X$ for each $i$,       
then, $G_{n}(t) \to G(t)$, $n \to \infty$, uniformly with respect to $t$.\\
(ii) Assume that the assumption in (i) is satisfied, and, 
\[ \sup_{i \in \{0, \dots, m-1\},  x \in X} \| Df_{n, i}(x) - Df_{i}(x) \| \to 0,\] 
then,   
\[ \liminf_{n \to \infty} \alpha_{n} \ge \alpha, \text{ and, } \liminf_{n \to \infty} \beta_n \ge \beta.\] 
(iii)  Assume that the assumptions in (i) and (ii) are satisfied, 
and, \[ \sup_{i, x} \|Df_i(x)^{-1}\| < +\infty,\]    
then, 
\[ \lim_{n \to \infty} \alpha_{n} = \alpha, \text{ and, } \lim_{n \to \infty} \beta_n = \beta.\]        
\end{Prop}

\begin{proof}
\[ G_{n}(t) - G(t) = f_{n, A_1(t)}(G_{n}(Ht)) - f_{A_1(t)}(G(Ht)). \]     
\[ = f_{n, A_1(t)}(G_n(Ht)) - f_{A_1(t)}(G_n(Ht)) + f_{A_1(t)}(G_n(Ht)) - f_{A_1(t)}(G(Ht)). \]

By using this and the assumption, 
\[ \limsup_{n \to \infty} |G_{n}(t) - G(t)| \le \limsup_{n \to \infty} \left|f_{A_1(t)}(G_{n}(Ht)) - f_{A_1(t)}(G(Ht))\right|. \]
Hence, 
\[ \limsup_{n \to \infty} \max_{t \in [0,1]} |G_{n}(t) - G(t)| \le \limsup_{n \to \infty} \max_{i \in \{0, \dots, m-1\}, t \in [0,1]} \left|f_{i}(G_{n}(t)) - f_{i}(G(t))\right|. \]
Since each $f_i$ is a weak contraction, we have assertion (i).   

Now we show (ii).  
By using the assumption, 
$\|Df_{n, A_1(t)}(G_n(Ht)) - Df_{A_1(t)}(G(Ht))\| \to 0$, as $n \to \infty$, uniformly with respect to $t$.    
Since \[ \sup_{i \in \{0, \dots, m-1\}, x \in X} \|Df_{i}(x)\| \le 1, \text{ and, }  \inf_{i \in \{0, \dots, m-1\}, x \in X} \|Df_{i}(x)^{-1}\| \ge 1, \] 
we have that for each $t$ and for all but finitely many positive integers $n$, 
\[ -\log_m \|Df_{n, A_1(t)}(G_n(Ht))\| \ge 0, \text{ and,  } \log_m \left\|Df_{n, A_1(t)}(G_n(Ht))^{-1} \right\| \le 0. \]   
By using them and Fatou's lemma,  
we have assertion (ii).

Now we show (iii).   
By using the assumption, we have that 
\[ \sup_{i \in \{0, \dots, m-1\}, x \in X}  -\log_m \|Df_i (x)\| \le  \sup_{i \in \{0, \dots, m-1\}, x \in X}  \log_m \|Df_i (x)^{-1}\| < +\infty.\]   
Therefore, \[ \log_m\|Df_{n, A_1(t)}(G_n(Ht))\| \to \log_m \|Df_{A_1(t)}(G(Ht))\|, \textup{ and, }  \]
\[ \log_m\|Df_{n, A_1(t)}(G_n(Ht))^{-1}\| \to \log_m \|Df_{A_1(t)}(G(Ht))^{-1}\|, \] 
as $n \to \infty$, uniformly with respect to $t$. 
Thus we have assertion (iii).   
\end{proof}

We assume many conditions in (iii), but some examples satisfy the conditions.    
\begin{Exa}[Perturbation]\label{perturb}
Let $X = [0,1]$ and $m=2$. Then, 
for sufficiently small $\epsilon$, \\
(i) If we let $f_{\epsilon, 0}(x) := x^{2}/2 - \epsilon x^{4}$ and $f_{\epsilon, 1}(x) := ((1 - \epsilon) x + (1 + \epsilon))/2$, 
then, $\alpha_{\epsilon} > 1$.\\  
(ii) If we let $f_{\epsilon, 0}(x) := x/(1+x) - \epsilon x(1-x)$ and $f_{\epsilon, 1}(x) := 1/(2-x) + \epsilon x^2(1-x)$,  
then, $\alpha_{\epsilon} > 1$.  
\end{Exa}


\section{Proofs}

\begin{Lem}\label{erg}
\begin{align*}
\left(\|(Df_{A_1(t)}(G(Ht)))^{-1}\|^{-1} + o(1)\right)M_{n}(Ht) &\le M_{n+1}(t) \\
&\le \left(\|Df_{A_1(t)}(G(Ht))\| + o(1)\right)M_{n}(Ht),  
\end{align*}
as $n \to \infty$.  
Here the small orders do not depend on $t$.   
\end{Lem}

\begin{proof}
We can rewrite (\ref{Def}) as 
\begin{equation}\label{rewrite}  
G(t) = f_{A_1(t)}(G(Ht)), \, t \in [0,1).  
\end{equation}

If $H(t_{n+1} + m^{-n-1}) > 0$, then, $Ht_{n+1} = (Ht)_n$ and $H(t_{n+1} + m^{-n-1}) = (Ht)_n + m^{-n}$.   
If $H(t_{n+1} + m^{-n-1}) = 0$, then, $Ht_{n+1} = (Ht)_n$ and $(Ht)_n + m^{-n} = 1$. 
By using them and (\ref{rewrite}),    
\begin{equation}\label{f2} 
M_{n+1}(t) = \left| f_{A_1(t)}(G((Ht)_n + m^{-n})) - f_{A_1(t)}(G((Ht)_n)) \right| \text{ and, }  
\end{equation}
\begin{equation}\label{f3}
M_{n}(Ht) = \left| G((Ht)_n + m^{-n}) - G((Ht)_n) \right|. 
\end{equation}

By using (\ref{f2}) and the mean value theorem,   
\begin{multline}\label{mv}
M_{n+1}(t) =  \Bigg| \int_{0}^{1} Df_{A_1(t)}\Big( (1-u) G((Ht)_n) + u G((Ht)_n + m^{-n})\Big) \\
 \Big(G((Ht)_n + m^{-n}) - G((Ht)_n) \Big) \ell(du) \Bigg| 
\end{multline}
Since $Df_{i}(x)$ is uniformly continuous with respect to $x$ (under the operator norm of $E$),    
\[ \lim_{n \to \infty} \max_{u \in [0,1]} \left\| Df_{A_1(t)}( (1-u) G((Ht)_n) + u G((Ht)_n + m^{-n})) - Df_{A_1(t)}(G(Ht)) \right\| = 0. \]
This convergence is uniform with respect to $t$.  
By using this, (\ref{f3}) and (\ref{mv}),   
we have the assertion. 
\end{proof}

Let $a \vee b$ be the maximum of real numbers $a$ and $b$.  

\begin{proof}[Proof of Theorem \ref{Basic}]  
Let $I(t) := \|Df_{A_1(t)}(G(Ht))\|$.  
Then, by using Lemma \ref{erg}, 
for any $\epsilon > 0$, there exists $k$ such that 
\[ M_{n}(t) \le M_{k}(H^{n-k}t) \prod_{i=0}^{n-k} I(H^i t)(1+\epsilon), \, \text{ for any } t \in [0,1)  \text{ and any } n > k.  \]
By using this and Birkoff's ergodic theorem (cf. Pollicott and Yuri \cite[Theorem 10.2]{PY}), 
for any $\delta > 0$,  
\begin{align*} 
\liminf_{n \to \infty} \frac{-\log_m M_n(t)}{n} 
&\ge \liminf_{n \to \infty} \frac{1}{n} \sum_{i=0}^{n-k} -\log_m (I(H^i t) \vee \delta) - \log_m (1+\epsilon).  \\
&\ge \int_{0}^{1} -\log_m (I \vee \delta) d\ell  - \log_m (1+\epsilon),  \, \text{ $\ell$-a.s.$t$.}  
\end{align*}

By using the monotone convergence theorem,   
\[ \lim_{\delta \to 0, \delta > 0} \int_{0}^{1} -\log_m (I \vee \delta) d\ell = \int_{0}^{1} -\log_m I(t) d\ell(t) = \alpha.\]      
By letting $\epsilon \to 0$, we have (\ref{lower}).     

We can show (\ref{upper}) in the same manner.  
\end{proof}

\begin{proof}[Proof of Theorem \ref{Dif}]
We first show assertion (ii).   
Assume $\beta < 1$. 
Then, Theorem \ref{Basic} implies that $m^n M_n(t) \to +\infty$, $n \to \infty$, $\ell$-a.s.$t$.    
\[ m^n M_n(t) \le \max\left\{ \frac{|G(t) - G(t_n)|}{t - t_n},  \frac{|G(t_n + m^{-n}) - G(t)|}{t_{n} + m^{-n} -  t} \right\}.\]
Hence, 
\[ \limsup_{s \to t} \frac{|G(t) - G(s)|}{|t-s|} \ge \limsup_{n \to \infty} m^n M_n(t) = +\infty, \, \text{ $\ell$-a.s. $t$}. \]
Thus we have (ii).

We second show assertion (i).  
Assume $\alpha > 1$. 
Let 
\[ I_{\epsilon}(t) := \sup\left\{\|Df_{A_1(t)}(x) \| : x \in X,  |x - G(Ht)| \le \epsilon \right\}, \, t \in [0,1]. \]     
This is a non-negative Borel measurable function since $E$ is separable.     
We remark that $| I_{\epsilon}(t)| \le 1$.   

By using the assumptions (A-iii) and (A-iv),   
$\lim_{\epsilon \to 0, \epsilon > 0} I_{\epsilon}(t) = I(t)$. 
By using this and the monotone convergence theorem,   
\[ \lim_{\epsilon \to 0, \epsilon > 0} \int_{[0,1)} -\log_m I_{\epsilon} d\ell  = \alpha. \]
By using this and $\alpha > 1$, 
there exists $\epsilon_0 > 0$ such that   
\begin{equation}\label{alpha} 
\int_{[0,1)} -\log_m I_{\epsilon_0} d\ell  > \frac{\alpha + 1}{2}.  
\end{equation} 
By using Birkoff's ergodic theorem again,  
\begin{equation}\label{conv-alpha} 
\lim_{n \to \infty} \frac{1}{n} \sum_{i=0}^{n-1} -\log_m I_{\epsilon_{0}} (H^i t)  =  \int_{[0,1)} -\log_m I_{\epsilon_0} d\ell,   \text{ $\ell$-a.s.$t$.} 
\end{equation} 
$\int_{[0,1)} -\log_m I_{\epsilon_0} d\ell$ can take $+\infty$, but the above convergence hold in the case.

Let $t$ be an $m$-normal number such that the above convergence holds. 
Here $t$ is $m$-normal means that $\lim_{n \to \infty} |\{k \in \{1, \dots, n\}: A_k(t) = i\}|/n = 1/m$ for each $i \in \{0, \dots, m-1\}$. 
Then, by noting (\ref{alpha}) and (\ref{conv-alpha}), 
there exists $M(t)$ such that   
\begin{equation}\label{alpha-est} 
\prod_{i=0}^{n-1} I_{\epsilon_{0}} (H^i t)   \le m^{-n(\alpha + 3)/4} \, \text{ for any $n \ge M(t)$.}       
\end{equation}

Let $n(t,s)$ be the minimum number such that $A_{n}(t) \ne A_n(s)$. 
Since $z$ is $m$-normal, $\lim_{s \to t} n(t,s) = +\infty$.  
Let 
\[ I(n, t, s) := \max_{u \in [0,1]} \left\| Df_{A_n(t)}(uG(H^n t) + (1-u)G(H^n s)) \right\|. \]
By using (\ref{rewrite}) and the mean value theorem,   
\begin{equation}\label{induc}
\left|G(H^{n-1} t) - G(H^{n-1} s)\right|  \le I(n, t, s) \left|G(H^n t) - G(H^n s)\right|, \, n \ge 1.  
\end{equation} 

Since 
\[ |H^n t - H^n s| \le m^{-n(t,s)^{1/2}}, \, \, \forall n \le n(t,s) - n(t,s)^{1/2},  \]   
$|G(H^n t) - G(H^n s)| \le \epsilon_0$ holds if $s$ is sufficiently close to $t$ (that is, $n(t,s)$ is sufficiently large).  
Therefore,    
\[ I(n, t, s)\le I_{\epsilon_0}( H^{n}t), \, \forall n \le n(t,s) - n(t,s)^{1/2}. \]

By using this and (\ref{alpha-est}), 
\[ \prod_{n=1}^{n(t,s)- n(t,s)^{1/2}} I(n, t, s) \le  m^{-(\frac{\alpha + 3}{4}) (n(t,s)- n(t,s)^{1/2})}. \]
holds if $s$ is sufficiently close to $t$. 
By using this and (\ref{induc}),               
\begin{align}\label{ts-upper}    
|G(t) - G(s)| &\le \left|G(H^{n(t,s)- n(t,s)^{1/2}} t) - G(H^{n(t,s)- n(t,s)^{1/2}} s) \right| \prod_{n=1}^{n(t,s)- n(t,s)^{1/2}} I(n, t, s) \notag \\   
&\le 2\max_{s \in [0,1]} |G(s)| \cdot m^{-(\frac{\alpha + 3}{4}) (n(t,s)- n(t,s)^{1/2} )}.  
\end{align}   

We give a lower bound for $|t-s|$. 
Let $N(t,s) > n(t,s)$ be the minimum number 
such that $A_{N(t,s)}(t) \ge 1$ if $t > s$ and $A_{N(t,s)}(t) \le m-2$ if $t < s$.        
Since $z$ is $m$-normal, 
we have that $\lim_{s \to t} N(t,s)/n(t,s) = 1$ and 
\[ |t - s| \ge m^{-N(t,s)} = m^{-n(t,s)(1 + o(1))}. \] 

By using this, (\ref{ts-upper}) and $\alpha > 1$, 
\[ \lim_{s \to t} \frac{|G(s) - G(t)|}{|s-t|} = 0,  \text{ $l$-a.s. $t$.} \] 
Thus we have assertion (i).  
\end{proof}

\begin{proof}[Proof of Corollary \ref{Var}] 
We have that 
\[ \sum_{k=1}^{m^{n}} \left| G\left(\frac{k}{m^n}\right) - G\left(\frac{k-1}{m^n}\right) \right|^p = m^{n} \int_{[0,1)} M_n^p d\ell. \]
Let $\epsilon > 0$ such that $p(\beta + \epsilon) < 1$. 
Thanks to (\ref{upper}),    
\[ \lim_{n \to \infty} \ell\left(M_n \ge m^{-n(\beta + \epsilon)}\right) = 1. \]
Hence, 
\[ m^{n} \int_{[0, 1)} M_n^p d\ell \ge m^{(1 - \beta - \epsilon)n}  \ell \left(M_n \ge m^{-n(\beta + \epsilon)}\right) \to +\infty, \, n \to \infty. \]  
\end{proof} 

\section{Remarks}  

\begin{Rem}   
(i) 
If $X = [0,1]$ and the solution $G$ is increasing and $\alpha > 1$, then, $\mu_{G}$ is a singular measure on $[0,1]$. 
However, we are not sure 
whether stronger results than singularity, 
for example, $\dim_{H}(\mu_{G}) < 1$, and, characterizing the points such that the derivative $DG$ is zero or infinity, etc.,  
can be shown by using Theorem \ref{Basic}. \\
(ii) 
Let $f_{\epsilon, 0}(x) = (1/2 +\epsilon) x$ and $f_{\epsilon, 1}(x) = (1/2 - \epsilon) x + 1/2 + \epsilon$.  
Let $G_{\epsilon}$ be the solution of (\ref{Def}) for $f_i = f_{\epsilon, i}$.    
Then, the partial derivative $\partial_{\epsilon} G_{\epsilon}(x)$ with respect to $\epsilon$ at $0$ gives the Takagi function. (Cf. Hata and Yamaguti \cite{HY}.)
It may be interesting to consider the partial derivative of $\partial_{\epsilon} G_{\epsilon}(x)$ for general $(f_{\epsilon, 0}, f_{\epsilon, 1})$.   
This may give a certain generalization of the Takagi function.\\ 
(iii) In Section 1, 
we assume that $X$ is a closed subset of a separable Banach space, 
but, if $X$ is a closed subset of a connected Riemannian manifold, 
we can define derivatives and variations of $G$. 
The corresponding regularity assumptions (A-iii) and (A-iv) for $f_i$ s are that $f_i$ s are $C^1$ maps between the interior of $X$. 
The Riemannian metric corresponds to the norm of $E$.  
\end{Rem}

\end{document}